\documentclass[12pt]{article}
\usepackage{amssymb}
\usepackage{amsthm}
\usepackage{amsmath}
\usepackage{dsfont}

\theoremstyle{plain}
\newtheorem{theorem}{theorem}[section]
\newtheorem{lemma}[theorem]{Lemma}

\newtheorem{conjecture}[theorem]{Conjecture}
\newtheorem{remark}[theorem]{Remark}

\title{Note on the diversity of intersecting families}
\date{}
\author{Xiaomei Chen\and Peng Jin\\{\footnote{\today}}
}
\begin{document}
\maketitle







\vspace{-15mm} 
\begin{abstract}
Let $n=2k+1$ be odd with $k\geq3$. In this note, we give two intersecting families with diversity larger than $\sum_{i=k+1}^{2k} \binom {2k}{i}$, which disprove a conjecture of Huang.
\end{abstract}




\section{Introduction}
Throughout the paper, we denote $[n]:=\{1,2,\ldots,n\}$, and $2^{[n]}:=\{X:X\subset [n]\}$. A family $\mathcal{F}\subset 2^{[n]}$ is called intersecting if any two of its elements intersect. The degree $\delta(x)$ of $x\in [n]$ is defined to be the number of members of $\mathcal{F}$ that contain $x$, and $\mathcal{F}$ is called regular if all of its elements have the same degree. If $\mathcal{F}$ is regular, we also denote $\sigma(\mathcal{F})$ the degree of elements of $\mathcal{F}$. Given an intersecting family $\mathcal{F}$, the diversity of $\mathcal{F}$ is defined to be $div(\mathcal{F}):=|\mathcal{F}|-\mathrm{max}_{x\in{[n]}}|\mathcal{F}(x)|$, where $\mathcal{F}(x)=\{F:x\in F\in \mathcal{F}\}$. Let $n=2k+1$ and $\mathcal{Q}_{k}={\{A: A\subset [2k+1],|A|\geq k+1\}}$. About the maximal diversity among families included in $2^{[n]}$, Huang\cite{[HHuang]} gave the following conjecture.
\begin{conjecture}
\label{conj:11}
For $n=2k+1$, suppose $\mathcal{F}\subset 2^{[n]}$ is intersecting. Then
\begin{equation*}
div(\mathcal{F})\leq div(\mathcal{Q}_{k})=\sum_{i=k+1}^{2k} \binom {2k}{i}
\end{equation*}
\end{conjecture}
In this note, we will give two regular intersecting families with diversity larger than $\mathcal{Q}_{k}$, which disprove the above conjecture.

\section{Counterexamples}
The first example is constructed from a regular intersecting family in $\binom {[n]}{k}$. The circular shift operation $\sigma$ on $[n]$ is defined as
$$\sigma(i)\equiv i+1  \textrm{ mod } n, \textrm{ for all entries } i=1,2,\ldots,n,$$
and we denote $\sigma(A)=\{\sigma(i):i\in A\}$ for $A\subset [n]$.

Let $L_{3}=\{1,2,4\}$ and $L_{k}=\{1,2,4\}\cup \{7,9,\ldots,2k-1\}$ for $k\geq 4$. For $k\geq 3$, we define $\mathcal{F}_{k}$ to be the family obtained by repeatedly applying $\sigma$ to $L_{k}$, i.e.
$$\mathcal{F}_{k}=\{\sigma^{i}(L_{k}):i=0,1,\ldots,n-1\}$$
For example, $\mathcal{F}_{3}$ consists of the 7 lines of the Fano plane.

\begin{lemma}
\label{lemma:21}
For $k\geq 3$ and $n=2k+1$, $\mathcal{F}_{k}\subset 2^{[n]}$ is a regular intersecting family with degree equal to $k$.
\end{lemma}
\begin{proof}
$\mathcal{F}_{3}$ consists of the 7 lines of the Fano plane, and thus is a regular intersecting family with degree 3. For $k>3$, it is clear that $\mathcal{F}_{k}$ is regular with degree $|\mathcal{F}_{k}|*k/n=k$. To prove that $\mathcal{F}_{k}$ is intersecting, we just need to show that $L_k$ and $\sigma^{i}(L_k)(1\leq i\leq n-1)$ intersect, since $\sigma^{n-i}(\sigma^{i}(L_k))=L_k$ and $\sigma^{n-i}(\sigma^{j}(L_k))=\sigma^{j-i}(L_k)$ for $1\leq i<j\leq n-1$. Assume that $L_k$ and $\sigma^{i}(L_k)(1\leq i\leq n-1)$ do not intersect for some $i\in [n-1]$, then we must have $\sigma^{i}(1)=5$ or $\sigma^{i}(1)=2k$. In the first case, we obtain $\sigma^{i}(2k-1)=2$, and $\sigma^{i}(4)=2$ in the second one. Both of the two cases contradict the assumption, thus $\mathcal{F}_{k}$ is intersecting.
\end{proof}

Given a set $A\subset [n]$, We denote $\bar{A}$ the complement of $A$ in $[n]$, and denote $\bar{\mathcal{F}}:=\{\bar{A}: A\in \mathcal{F}\}$ for $\mathcal{F}\subset 2^{[n]}$. The family $\mathcal{P}_k$ is defined as
$$\mathcal{P}_k=\mathcal{F}_k\cup (\mathcal{Q}_k\setminus \bar{\mathcal{F}_k})$$

\begin{theorem}
\label{thm:21}
For $n=2k+1$ with $k\geq 3$, $\mathcal{P}_k$ is a regular intersecting family and $div(\mathcal{P}_k)=div(\mathcal{Q}_k)+1$.
\end{theorem}
\begin{proof}
By Lemma \ref{lemma:21}, the family $\mathcal{F}_k$ is intersecting. If $\mathcal{P}_k$ is not intersecting, then there exist disjoint sets $F$ and $Q$ with $F\in \mathcal{F}_k$ and $Q\in \mathcal{Q}_k\setminus \bar{\mathcal{F}_k}$. Since $F\cap Q=\emptyset$ and $|Q|\geq |\bar{F}|=k+1$, we must have $Q=\bar{F}$, which contradicts the choice of $F$ and $Q$. Therefore $\mathcal{P}_k$ is intersecting.

Since $\mathcal{F}_k$, $\bar{\mathcal{F}_k}$ and $\mathcal{Q}_k$ are regular with degree equal to $k$, $k+1$ and $\sum_{i=k}^{2k}\binom{2k}{i}$ respectively, and $|\mathcal{P}_k|=|\mathcal{Q}_k|=\sum_{i=k+1}^{2k+1}\binom{2k+1}{i}$, we have
$$div(\mathcal{P}_k)=|\mathcal{P}_k|-\sigma(\mathcal{F}_k)-(\sigma(\mathcal{Q}_k)-\sigma(\bar{\mathcal{F}_k}))=\sum_{i=k+1}^{2k}\binom{2k}{i}+1=div(\mathcal{Q}_k)+1$$
\end{proof}

The second example is related to finite projective planes. Let $q$ be an odd prime power, and let $\mathds{P}$ be the transitive projective plane $\mathds{P}^2(\mathds{F}_q)$ over the finite field $\mathds{F}_q$. We take $n=q^2+q+1$, and identify the set of points of $\mathds{P}$ with $[n]$. For $q+1 \leq i\leq \frac{q^2+q}{2}$, we define $\mathcal{A}_i$ to be the family of all $i-$element subsets consisting of the points of $\mathds{P}$ that contain a line of $\mathds{P}$.
Since $\mathds{P}$ is transitive, $\mathcal{A}_i$ is a regular intersecting family for $i\geq q+1$. See \cite{[Ellis1]} for more details about the family $\mathcal{A}_i$.

Let $k=\frac{q^2+q}{2}$ and $\mathcal{A}=\bigcup_{i=q+1}^{k}\mathcal{A}_i$. We define the family $R_k$ as
$$\mathcal{R}_k=\mathcal{A}\cup (\mathcal{Q}_k\setminus \mathcal{\bar{A}}).$$
Then the following result will show that $\mathcal{R}_k$ is a family in $2^{[n]}$ with larger diversity than $\mathcal{Q}_k$.
\begin{theorem}
\label{thm:22}
For an odd prime power $q$, let $n=q^2+q+1$ and $k=\frac{q^2+q}{2}$. Then $\mathcal{R}_k$ is a regular intersecting family, and we have
$$div(\mathcal{R}_k)-div(\mathcal{Q}_k)>\sum_{i=q+1}^{k}\frac{n-2i}{2}\binom{n-q-1}{i-q-1}$$
\end{theorem}

\begin{proof}
If $\mathcal{R}_k$ is not intersecting, then there exist disjoint sets $A$ and $Q$ with $A\in \mathcal{A}$ and $Q\in \mathcal{Q}_k\setminus \mathcal{\bar{A}}$. Since $A\cap Q=\emptyset$ and $|Q|\leq k$, we have $A\subset \bar{Q}\in 2^{[n]}\setminus \mathcal{Q}_k$. On the other hand, we know that $\mathcal{R}_k$ is an upset by its definition , i.e. is closed under taking supersets. Thus we must have $\bar{Q}\in \mathcal{A}$, which contradicts the choice of $Q$. Therefore $\mathcal{R}_k$ is an intersecting family.

It is clear that $\mathcal{A}_i$ and $\bar{\mathcal{A}_i}$ are regular, so $\mathcal{R}_k$ is also regular. To obtain the diversity of $\mathcal{R}_k$, we firstly compute the size of $\mathcal{A}_i$. Since $\mathcal{A}_{q+1}$ consists of all lines of the projective plane, we have $|\mathcal{A}_{q+1}|=n$. For $q+1<i\leq k$, by using the Bonferroni inequalities, we have
\begin{align}
\label{equ:21}
\mathcal{A}_i\geq n\binom {n-q-1}{i-q-1}-\binom{n}{2}\binom{n-2q-1}{i-2q-1}\geq \frac{n}{2}\binom{n-q-1}{i-q-1}
\end{align}

Since $|\mathcal{R}_k|=|\mathcal{Q}_k|, \sigma(\mathcal{A}_i)+\sigma(\bar{\mathcal{A}}_i)=|\mathcal{A}_i|$, we have
\begin{align}
\label{equ:22}
div(\mathcal{R}_k)&=|\mathcal{R}_k|-\sigma(\mathcal{A})-(\sigma(\mathcal{Q}_k)-\sigma(\bar{\mathcal{A}}))\nonumber\\
&=div(\mathcal{Q}_k)+\sum_{i=q+1}^{k}(|\mathcal{A}_i|-2\sigma(\mathcal{A}_i)),
\end{align}
Combining (\ref{equ:21}) and (\ref{equ:22}) together, we have
\begin{align*}
div(\mathcal{R}_k)-div(\mathcal{Q}_k)&=\sum_{i=q+1}^{k}(|\mathcal{A}_i|(1-\frac{2i}{n}))\\
&>\sum_{i=q+1}^{k}\frac{n-2i}{2}\binom{n-q-1}{i-q-1}.
\end{align*}
\end{proof}

\begin{remark}
\rm{By the construction of the above example, to get the maximum diversity among families included in $2^{[n]}$, it will be meaningful to study the existence of regular intersecting families included in $\binom{[n]}{k}$ and the maximum size of such families when $k$ is relatively small compared to $[n]$. Readers could refer to \cite{[Ihringer]} for more information about regular intersecting families.}
\end{remark}





\noindent{\emph{Address}: School of Mathematics and Computational Science, Hunan University of Science and Technology, Xiangtan 411201, China.}\\
\noindent{\emph{E-mail address}: xmchen@hnust.edu.cn, jin\_peng10@163.com}
\end{document}